\documentclass[a4paper,reqno,12pt]{amsart}

\usepackage{amsmath, amssymb, amsthm, enumerate,amsfonts,latexsym, mathrsfs}
\usepackage{lipsum}

\usepackage[top=30mm,right=27mm,bottom=30mm,left=27mm]{geometry}
\newtheorem{theorem}{Theorem}[section]

\newtheorem{lemma}[theorem]{Lemma}
\newtheorem{proposition}[theorem]{Proposition}

\makeatletter
\newcommand{\imod}[1]{\allowbreak\mkern4mu({\operator@font mod}\,\,#1)}
\makeatother

\newtheorem*{thmA}{Theorem~A}
\newtheorem*{thmB}{Theorem~B}
\newtheorem*{thmC}{Theorem~C}
\newtheorem*{thmD}{Theorem~D}

\newcommand{\Irr}{{\mathrm {Irr}}}
\newcommand{\cd}{{\mathrm {cd}}}
\newcommand{\Aut}{{\mathrm {Aut}}}

\newcommand{\PSL}{{\mathrm {PSL}}}

\newcommand{\SSS}{\mathrm{S}}
\newcommand{\A}{\mathrm{A}}

\newcommand{\Atlas}{{\sf Atlas}}
\newcommand{\GAP}{{\sf GAP}}
\newcommand{\Van}{{\mathrm {Van}}}
\newcommand{\ord}{{\mathrm {ord}}}

\newcommand{\Vo}{{\mathrm {Vo}}}

\theoremstyle{definition}

\begin{document}
\title[\textbf{Orders of  vanishing elements}]{\textbf{On the orders of vanishing elements of finite groups}}

\author{Sesuai Y. Madanha}
\address{Department of Mathematics and Applied Mathematics, University of Pretoria, Private Bag X20, Hatfield, Pretoria 0028, South Africa}
\email{sesuai.madanha@up.ac.za}


\date{\today}


\begin{abstract}
Let $ G$ be a finite group and $p$ be a prime. Let $ \Vo(G) $ denote the set of the orders of vanishing elements,  $\Vo_{p} (G)$ be the subset of  $ \Vo(G) $ consisting of those orders of vanishing elements divisible by $p$ and $\Vo_{p'} (G) $ be the subset of $ \Vo(G) $ consisting of those orders of vanishing elements not divisible by $p$. Dolfi, Pacifi, Sanus and Spiga proved that if $ a $ is not a $ p $-power for all $ a\in \Vo(G)$, then $ G $ has a normal Sylow $ p $-subgroup. In another article, the same authors also show that if if $ \Vo_{p'}(G) =\emptyset $, then $ G $ has a normal nilpotent $ p $-complement. These results are variations of the well known Ito-Michler and Thompson theorems. In this article we study solvable groups such that $|\Vo_{p}(G)| = 1 $ and show that $ P' $ is subnormal. This is analogous to the work of Isaacs, Mor\'eto, Navarro and Tiep where they considered groups with just one character degree divisible by $ p $. We also study certain finite groups $G$ such that $|\Vo_{p'}(G)| = 1 $ and we prove that $ G $ has a normal subgroup $ L $ such that $ G/L $ a normal $ p $-complement and $ L $ has a normal $ p $-complement.  This is analogous to the recent work of Giannelli, Rizo and Schaeffer Fry on character degrees with a few $p'$-character degrees. Bubboloni, Dolfi and Spiga studied finite groups such that every vanishing element is of order $ p^{m} $ for some integer $ m\geqslant 1 $. As a generalization, we investigate groups such that $ \gcd(a,b)=p^{m} $ for some integer $ m \geqslant 0 $, for all $ a,b\in \Vo(G) $. We also study finite solvable groups whose irreducible characters vanish only on elements of prime power order. 
\end{abstract}

\maketitle

%%%%%%%%%%%%%%%%%%%%%%%%%%%%%%%%%%%%%%%%%%%%%%%%%
%----------------------------------------------------------------------------

`\section{Introduction}
One of the interesting problems in character theory of finite groups is determine the structure a group using information from its character table. Below are some results to this effect. Let $G$ be a finite group and $p$ be a prime. Let $\Irr(G)$ denote the set of complex irreducible characters of $G$, $\cd(G)$ the set of character degrees of $G$, $ \cd_{p}(G)$ the subset of $\cd(G) $ consisting of character degrees divisible by $p$ and $ \cd_{p'}(G)$ the subset of $\cd(G) $ consisting of those character degrees not divisible by $p$. A classical theorem of Thompson \cite[Corollary 12.2]{Isa06} states that if $ |\cd_{p'}(G)| = 1$, then $ \textbf{O}^{pp'}(G) = 1$, where $\textbf{O}^{pp'}(G) = \textbf{O}^{p'}(\textbf{O}^{p}(G))$. Let $ \textbf{O}^{pp'pp'}(G) = \textbf{O}^{pp'}(\textbf{O}^{pp'}(G))$. Recently, Giannelli, Rizo and Schaeffer Fry \cite{GRS20} proved a variation Thompson's theorem by showing the following result:

\begin{theorem}\cite[Theorem A]{GRS20}\label{GRS20Theorem}
Let G be a finite group and let $ p > 3 $ be a prime. Suppose that $|\cd_{p'}(G)| = 2 $. Then G is solvable and $ \textbf{O}^{pp'pp'}(G) = 1$.
\end{theorem}

We are interested in studying sets of vanishing elements with some restrictions corresponding to these results. An element $ g \in G $ is a vanishing element of $ G $ if there exists $\chi \in \Irr(G) $ such that $\chi(g) = 0 $. The set of all vanishing elements of G is denoted by Van(G). A classical theorem of Burnside \cite[Theorem 3.15]{Isa06} states that $\Van(G)$ is non-empty if there is a non-linear $\chi \in \Irr(G)$. Let $ \Vo(G) $ denote the set of the orders of elements in $ \Van(G) $, $ \Vo_{p}(G) $ be the subset of $ \Vo(G) $ consisting of orders of vanishing elements divisible by $p$ and $ \Vo_{p'}(G) $ be the subset of $ \Vo(G) $ consisting of those orders of vanishing elements not divisible by $p$.

Dolfi Pacifi, Sanus and Spiga proved the following result analogous to Thompson's Theorem:

\begin{theorem}\cite[Corollary B]{DPSS09}\label{normalp-complementvanishing} Let $ G $ be a finite group and let $p$ be a prime. Suppose that $ \Vo_{p'}(G) =\emptyset $. Then $ \textbf{O}^{pp'}(G)=1$.
\end{theorem}

As a variation of Theorem \ref{normalp-complementvanishing}, we study certain finite groups $G$ such that $|\Vo_{p'}(G)| = 1 $ with an extra condition and obtain our first result.

\begin{thmA}\label{thm A}
Let $G$ be a finite group and let $p $ and $ q $ be distinct primes. Suppose that $|\Vo_{p'}(G)| = 1$. Then G is solvable. Moreover, suppose that one of the following holds:
\begin{itemize}
\item[(a)] $ \Vo_{p'}(G)=\{b\} $ such that $ b $ is divisible by at least two primes;
\item[(b)] $ \Vo_{p'}(G)=\{q^{n}\} $ for some positive integer $ n $ and $ Q' $ is subnormal, where $ Q $ is a Sylow $ q $-subgroup of $ G $.
\end{itemize}
Then $\textbf{O}^{pp'pp'}(G) = 1 $.
\end{thmA}

We remark that it is not clear if the hypothesis that $ Q' $ is subnormal in $ G $ is necessary. In our proof of Theorem A, we use a theorem of Franciosi, de Giovanni, Heineken and Newell (see Lemma \ref{FdGHN91Lemma3}) which states that if $ G=AB $ is a product of an abelian group $ A $ and a nilpotent group $ B $, then $ A\mathbf{F}(G) $ is normal in $ G $ and $ G $ has Fitting height at most $ 3 $. In general, it is well known that for any positive integer $ n $, there is group of order $ p^{a}q^{b} $ and Fitting height $ n $. Hence a Fitting height of a product of two Sylow subgroups can be arbitrarily large.

The opposite of Thompson's theorem is the famous Ito-Michler theorem. It states that if $ \cd_{p}(G)=\emptyset $, then $ G $ has a normal and abelian Sylow $ p $-subgroup. As a generalization of this, Isaacs, Mor\'eto, Navarro and Tiep \cite[Theorem A]{IMNT09} proved that if $ |\cd_{p}(G)|=1 $, then $ P' $ is abelian and subnormal in $ G $. In particular, $ P/\textbf{O}_{p}(G) $ is cyclic when $ G $ is solvable. Another generalization of the Ito-Michler theorem in terms of vanishing elements was proved by Dolfi, Pacifi, Sanus and Spiga \cite{DPSS09}. They showed that if all $ p $-elements are non-vanishing elements of $ G $, then $ G $ has a normal Sylow $ p $-subgroup and that the Sylow $ p $-subgroup need not be abelian. Their result implies that if $ \Vo_{p}(G)=\emptyset $, then $ G $ has a normal Sylow $ p $-subgroup. As an analogue of the theorem of Isaacs, Mor\'eto, Navarro and Tiep \cite[Theorem A(b)]{IMNT09} for solvable groups, we consider the following for solvable groups. 

\begin{thmB}\label{thmB}
Let $ G $ be a finite solvable group. If $ |\Vo_{p}(G)|=1 $, then $ P' $ is subnormal in $ G $. In particular, $ P/\textbf{O}_{p}(G) $ is cyclic.
\end{thmB}

We now discuss our next problem. Bubboloni, Dolfi and Spiga studied finite groups such that every vanishing element is of order $p^{m}$ for some integer $ m > 1  $. They proved the following result:

\begin{theorem}\cite[Corollary B]{BDS09}\label{normalp-complementvanishing2}
Let $ G $ be a finite group and $ p $ be a prime. If $a =p^{m} $ for some integer $ m \geqslant 1 $, for all $ a\in \Vo(G)$, then one of the following holds:
\begin{itemize}
\item[(a)] $ G $ is a $ p $-group.
\item[(b)] $ G/\mathbf{Z}(G) $ is a Frobenius group with a Frobenius complement of $ p $-power order and $ \mathbf{Z}(G)=\textbf{O}_{p}(G) $. 
\end{itemize}
\end{theorem}

Observe that groups described in Theorem \ref{normalp-complementvanishing2} are such that $ \textbf{O}^{pp'}(G)=1 $. Consider finite groups $G$ with the following extremal property for a fixed prime $p$:

\[\mbox{\emph{$ \gcd(a,b)=p^{m} $ for some integer $ m \geqslant 0 $, for all $ a,b\in \Vo(G) $.} \label{ee:star} \tag{$\star$}}\] 
\\
Property \eqref{ee:star} is a generalization of the property in Theorem \ref{normalp-complementvanishing2}. For if $ a =p^{m} $ for some integer $ m \geqslant 1 $, for all $ a\in \Vo(G) $, then $ \gcd(a,b)=p^{k} $ for some integer $ k\geqslant 1 $, for all $a,b \in \Vo(G) $. Conversely, $ \mathrm{S}_{4} $ and $ \mathrm{A}_{5} $ are examples of groups that satisfy property \eqref{ee:star} for $ p=2 $ but do not satisfy the property in Theorem \ref{normalp-complementvanishing2}. Property \eqref{ee:star} is also a generalization of a property studied in \cite{MR20}.

Using the classification of finite simple groups, we show that if $G$ satisfies property \eqref{ee:star} for an odd prime $ p > 7 $, then $ \textbf{O}^{pp'pp'}(G)=1 $:

\begin{thmC}\label{thmC}
Let $ G $ be a finite group and $ p > 7 $ be a prime. If $ G $ satisfies property \eqref{ee:star} for $ p $, then $ G $ is solvable and $ \textbf{O}^{pp'pp'}(G)=1 $.
\end{thmC}

The condition that $ p > 7 $ is necessary in Theorem C : $ \mathrm{A}_{7} $ shows the theorem fails for any $ p \leqslant 7 $.

Zhang, Li and Shao \cite{ZLS11} studied finite non-solvable groups whose irreducible characters vanish only on elements of prime power order. We investigate finite solvable groups whose irreducible characters vanish only on elements of prime power order. This is another generalization of Theorem \ref{normalp-complementvanishing2}. We show that either these groups have vanishing elements which are of $ p $-power order or they are groups of order $ p^{a}q^{b} $ for some primes $ p $ and $ q $. 

\begin{thmD}\label{thmD}
Let $ G $ be a finite non-abelian solvable group and $ p, q $ be distinct primes. If every irreducible character of $ G $ vanishes on elements of prime power order only, then one of the following holds.
\begin{itemize}
\item[(i)] $ G $ is a $ p $-group;
\item[(ii)] $ G/\textbf{Z}(G) $ is a Frobenius group with a Frobenius complement of $ p $-power order and $ \textbf{Z}(G)=\textbf{O}_{p}(G) $;
\item[(iii)] $ G $ is a Frobenius group with a Frobenius complement of $ p $-power order and a kernel of $ q $-power order, that is, $ \pi(G)=\{p,q\} $;
\item[(iv)] $ G $ is a nearly $ 2 $-Frobenius group and $ \pi(G)=\{p,q\} $.
\end{itemize}
\end{thmD}

Higman \cite{Hig57} described finite groups in which every element is of prime power order. In particular, he proved that the solvable groups with this property have orders that are divisible by at most two distinct primes. Hence Theorem D generalizes Higman's result for solvable groups. Note that groups in Theorem D(ii) have orders that may be divisible by more than two distinct primes (see \cite[Proposition 5.4]{BDS09}).

\section{Preliminaries}
In this section we shall list some properties of vanishing elements.
%\begin{lemma}\cite[Lemma 2.1]{MR20}\label{G/Nsatifies}
%Let $ G $ be a finite group and let $ N $ be a normal subgroup of $ G $. If $ G $ satisfies property \eqref{ee:star}, then $ G/N $ satisfies property \eqref{ee:star}.
%\end{lemma}

\begin{lemma}\cite[Lemma 2]{Qia02}\label{QiaLemma2}
Let $ G $ be a finite solvable group. Suppose $ M, N $ are normal subgroups of $ G $.
\begin{itemize}
\item[(a)] If $ M\setminus N $ is a conjugacy class of $ G $ and $ \gcd(|M{:}N|, |N|)=1 $, then $ M $ is a Frobenius group with kernel $ N $ and prime order complement.
\item[(b)] If $ G\setminus N $ is a conjugacy class of $ G $, then $ G $ is a Frobenius group with an abelian kernel and complement of order two.
\end{itemize} 
\end{lemma}

%\begin{proposition}\cite[Proposition 2.5]{DPSS10a}
%Let $G$ be a group, and let $M$ and $N$ be normal subgroups of $G$ such that $M \cap N = 1 $. Assume that there is an element $ m \in M \cap \Van(G) $. Then $ mn \in \Van(G) $ for all $ n \in N $.
%\end{proposition}
%Let $ m\geqslant 2 $ be a positive integer. Recall that $ \pi(m):=\{p\mid p $ divides $ m $, where $ p $ is prime$ \} $.
%
%\begin{corollary}\cite[Corollary 2.6]{DPSS10a}\label{DPSS10aCorollary2.6}
%Let $G$ be a finite group and let $K$ be a nilpotent normal subgroup of $G$. If $ K \cap \Van(G) \neq \emptyset $, then there exists $ g\in K\cap \Van(G)$ whose order is divisible by every prime in $ \pi(|K|) $.
%\end{corollary}
%
%\begin{lemma}\cite[Proposition 2.1]{DPSS10b}
%Let $G$ be a group, and assume that $ \mathbf{F}(G) $ contains an element of $ \Van(G) $. Then there exists $ g\in \mathbf{F}(G)\cap \Van(G) $ such that $ \pi (\ord(g))=\pi (\mathbf{F}(G)) $.
%\end{lemma}
%
%\begin{lemma} \cite[Theorem D]{INW99}\label{INW99TheoremD} Let G be a finite solvable group. If $x$ is a non-vanishing element of $G$, then $ x\mathbf{F}(G) $ is a $2$-element of $ G/\mathbf{F}(G) $. If $ G $ is not nilpotent, then $ x $ lies in the penultimate term of the Fitting series.
%\end{lemma}
%
%A non-linear irreducible character $ \chi $ of G is said to be of $ p $-defect zero if $ p $ does not divide $ |G|/\chi(1) $. By a result of Brauer (see \cite[Theorem 8.17]{Isa06}), if $ \chi $ is an irreducible character of $ p $-defect zero of $ G $, then $ \chi(g) = 0 $ whenever $ p $ divides the order of $ g $ in $ G $.
The existence of $ p $-defect zero characters is guaranteed in finite simple groups $ G $ for all primes $ p\geqslant 5 $  dividing $ |G| $ as the following result shows:

\begin{lemma}\cite[Corollary 2.2]{GO96}\label{pdefectzero5ormore} Let $ G $ be a non-abelian finite simple group and $ p $ be a prime. If $ G $ is a finite group of Lie type, or if $ p\geqslant 5 $, then there exists $ \chi \in \Irr(G) $ of $ p $-defect zero. 
\end{lemma}

\begin{lemma}\cite[Lemma 2.2]{Bro16}\label{Bro16Lemma2.2} Let $ G $ be a finite group, $ N $ a normal subgroup of $ G $ and $ p $ a prime. If $ N $ has an irreducible character of $ p $-defect zero, then every element of $ N $ of order divisible by $ p $ is a vanishing element in $ G $.
\end{lemma}

\begin{lemma}\cite[Lemma 5]{BCLP07}\label{BCLP07Lemma5} Let $ G $ be a finite group, and $ N=S_{1}\times \cdots \times S_{k} $ a minimal normal subgroup of $ G $, where every $ S_{i} $ is isomorphic to a non-abelian simple group $ S $. If $ \theta \in \Irr(S) $ extends to $ \Aut(S) $, then $ \varphi= \theta \times \cdots \times \theta \in \Irr(N) $ extends to $ G $.
\end{lemma}

\begin{lemma}\cite[Theorem 1.1]{MT-V11}\label{MT-V11Theorem1.1}
Suppose that $ N $ is a minimal normal non-abelian subgroup of a finite group $ G $. Then there exists an irreducible character $ \theta $ of $ N $ such that $ \theta $ is extendible to $ G $ with $ \theta(1)\geqslant 5 $.
\end{lemma}

Given a finite set of positive integers $Y,$ the prime graph $ \Pi (Y) $ is defined as the undirected graph whose vertices are the primes $ p $ such that there exists an element of $ Y $ divisible by $ p $, and two distinct vertices $ p, q $ are adjacent if and only if there exists an element of $ Y $ divisible by $ pq $. The vanishing prime graph of $ G $, denoted by $ \Gamma(G) $, is the prime graph $ \Pi(\mathrm{Vo}(G)) $. For a group $ G $, let $ \omega(G) $ be the set of orders of elements of $ G $. We shall state a result on non-solvable groups with disconnected vanishing prime graphs. Let $ n(\mathcal{G}) $ be the number of connected components of the graph $ \mathcal{G} $.

\begin{theorem}\cite[Theorem B]{DPSS10a}\label{DPSS10aTheoremB}
Let $ G $ be a finite non-solvable group. If $ \Gamma(G) $ is disconnected, then $ G $ has a unique non-abelian composition factor $ S $,  and $ n(\Gamma(G)) \leqslant n(\Pi(\omega(S))) $ unless $ G $ is isomorphic to $ \A_{7} $.
\end{theorem}

We shall also state a result on solvable groups with disconnected vanishing prime graphs. We first recall two definitions: 

A group $ G $ is said to be a \emph{$ 2 $-Frobenius group} if there exists two normal subgroups $ F $ and $ L $ of $ G $ such that $ G/F $ is a Frobenius group with kernel $ L/F $ and $ L $ is a Frobenius group with kernel $ F $. 

A group $ G $ is said to be a \emph{nearly $ 2 $-Frobenius group} if there exist two normal subgroups $ F $ and $ L $ of $ G $ with the following properties: $ F=F_{1}\times F_{2} $ is nilpotent, where $ F_{1} $ and $ F_{2} $ are normal subgroups of $ G $. Furthermore, $ G/F $ is a Frobenius group with kernel $ L/F $, $ G/F_{1} $ is a Frobenius group with kernel $ L/F_{1} $, and $ G/F_{2} $ is a $ 2 $-Frobenius group.

\begin{theorem} \cite[Theorem A]{DPSS10b} \label{DPSS10bTheoremA}
Let $ G $ be a finite solvable group. Then $ \Gamma(G) $ has at most two connected components. Moreover, if $ \Gamma(G) $ is disconnected, then $ G $ is either a Frobenius group or a nearly $ 2 $-Frobenius group.
\end{theorem}

The following is a classification of Frobenius complements.

\begin{theorem}\cite[Theorem 1.4]{Bro01}\label{Bro01Theorem1.4}
Let $ G $ be a Frobenius group with Frobenius complement $ M $. Then $ M $ has a normal subgroup $ N $ such that all Sylow subgroups of $ N $ are cyclic and one of the following holds:
\begin{itemize}
\item[(a)] $ M/N \cong 1 $;
\item[(b)] $ M/N \cong \mathrm{V}_{4} $, the Sylow $ 2 $-subgroup of the alternating group $ \mathrm{A}_{4} $;
\item[(c)] $ M/N \cong \mathrm{A}_{4} $;
\item[(d)] $ M/N \cong \mathrm{S}_{4} $;
\item[(e)] $ M/N \cong \mathrm{A}_{5} $;
\item[(f)] $ M/N \cong \mathrm{S}_{5} $.
\end{itemize}
\end{theorem}

\begin{lemma}\cite[Lemma 3]{FdGHN91}\label{FdGHN91Lemma3}
Let the finite group $G = AB$ be the product of an abelian subgroup $A$ and a nilpotent subgroup $B$. Then $A\mathbf{F}(G)$ is a normal subgroup of $G$. In particular, the Fitting height of $ G $ is at most $ 3 $.
\end{lemma}

\section{A reduction theorem}

 We shall prove a reduction theorem in this section. We shall first prove the proposition below. Let $ \mathfrak{S}=\{ \A_{5},  \A_{6}, \A_{7}, \PSL_{2}(7), \PSL_{2}(8), \PSL_{3}(4), ^{2}\!\mathrm{B}_{2}(8) \} $.

\begin{proposition}\label{vanishingordersaremore} Let $ N $ be a non-abelian minimal normal subgroup of $ G $ such that $ N\notin \mathfrak{S} $. Then one of the following holds
\begin{itemize}
\item[(a)] There are distinct primes $ p_{1},p_{2}\mid |N| $ and vanishing elements $ x,y,z $ of $ G $ contained in $ N $ such that $ \ord(x)=p_{1} $, $ \ord(y)=p_{2} $ and $ p_{1}p_{2}\mid \ord(z) $.
\item[(b)] There are distinct primes $ p_{1},p_{2} \mid |N| $ and vanishing elements $ x_{1},x_{2},y_{1}, y_{2} $ of $ G $ such that $ \ord(x_{i})\not= \ord(y_{i}) $, $ p_{i}\mid \ord(x_{i}) $ and $ p_{i}\mid \ord(y_{i}) $ for $ i=1,2 $.
\item[(c)] There is a prime $ p \neq 3 $ such that $ p \mid |N| $ and vanishing elements $ x,y,z $ of $ G $ such that $ \ord(x)=3 $, $ \ord(y)=3p $ and $ \ord(z)=p^{2} $.
\item[(d)] There is an odd prime $ p\mid |N| $ and vanishing elements $ x,y,z $ of $ G $ such that $ \ord(x)=p $, $\ord(y)=2p $ and $ \ord(z)=4p $.
\end{itemize}
\end{proposition}
\begin{proof}
Let $ N=S_{1}\times S_{2}\times \cdots \times S_{k} $,  where $ S_{i}\cong S $, $ S $ is a simple group. Suppose that $ k\geqslant 2 $.  Using \cite[Theorem 3.10]{Isa06}, $ \pi(|S|)\geqslant 3 $. Suppose that $ \pi(|S|)\geqslant 4 $. Then there are distinct primes $ p_{1}, p_{2}\geqslant 5 $ such that $ p_{i} \mid |N| $ and so $ N $ has an irreducible character of $ p_{i} $-defect zero for $ i=1,2 $ by Lemma \ref{pdefectzero5ormore}. Let $ x_{i}\in S_{i} $ be a $ p_{i} $-element for $ i=1,2 $. Note that $ x_{i} $ is a vanishing element of $ G $ by Lemma \ref{Bro16Lemma2.2}. The element $ x_{1}x_{2} $ is a vanishing element and $ \ord(x_{1}x_{2}) $ is divisible by $ p_{1}p_{2} $. This is case (a).

We may assume that $ |\pi(G)|=3 $. Using a result of Herzog in \cite{Her68}, we have that $ S\in \{\PSL_{2}(5), \PSL_{2}(7), \PSL_{2}(8), \PSL_{2}(9), \PSL_{2}(17), \PSL_{3}(3), \mathrm{PSU}_{3}(3), \mathrm{PSU}_{4}(2)\} $. In particular, $ S $ is a simple group of Lie type. By Lemma \ref{pdefectzero5ormore}, $ S $ and hence by extension $ N $ has an irreducible character of $ p $-defect zero for all $ p\mid |N| $. Let $ x_{1}\in S_{1} $ be a $ 2 $-element and $ x_{2}\in S_{2} $ be a $ p $-element for some odd prime $ p $. Now $ x_{1} $, $ x_{2} $ and $ x_{1}x_{2} $ are vanishing elements and the order of $ x_{1}x_{2} $ is divisible by $ 2p $. This is case (a).

We may assume that $ N$ is a simple group. Let $ N$ be a sporadic simple group or $ ^{2}\mathrm{F}_{4}(2)' $. Table \ref{SporadicTable}.1 below contains an irreducible characters $ \theta_{1} $ and $ \theta_{2} $ of $ N $ of $ p $-defect zero and $ q $-defect zero for some primes $ p $ and $ q $, respectively, and three elements of distinct orders divisible by $ p $, $ q $ and $ pq $, respectively. This is case (a) and it follows from Lemma \ref{Bro16Lemma2.2}. We shall use the character tables and notation in the \Atlas{} \cite{CCNPW85}.

\begin{center}\label{SporadicTable}
Table 3.1
\begin{tabular}{|c|c|c|c|c|c|}
\cline{1-6}
$ N $ & $ \theta_{1}(1) $ & $ \theta_{2}(1) $ & $ \ord(v\mathcal{C}_{1}) $  & $ \ord(v\mathcal{C}_{2}) $ & $ \ord(v\mathcal{C}_{3}) $ \\
\cline{1-6}
$ M_{11} $  &   $ \chi_{9}(1)=45 $ & $ \chi_{6}(1)=16 $  &   $ 3A $ & $ 2A $   &  $ 6A $  \\
\cline{1-6}
$ M_{23} $  &   $ \chi_{3}(1)=45 $ &  $ \chi_{12}=896 $ &   $ 3A $ & $ 2A $  &  $ 6A $ \\
\cline{1-6}
$ M^{c}L $  &   $ \chi_{9}(1)=1750 $ & $ \chi_{7}=896 $  &   $ 5A $ & $ 2A $   &  $ 10A $\\
\cline{1-6}
$ Co_{2} $  &   $ \chi_{5}(1)=1771 $  & $ \chi_{54}=1835008 $ &   $ 7A $  & $ 2A $  &  $ 14A $ \\
\cline{1-6}
$ He $  &   $ \chi_{9}(1)=1275 $ & $ \chi_{54}(1)=21504 $ &   $ 5A $ & $ 2A $   &  $ 10A $ \\
\cline{1-6}
 $ Fi_{22} $ & $ \chi_{4}(1)=1001 $ &  $ \chi_{56}(1)=1441792 $ & $ 7A $ & $ 2A $  & $ 14A $ \\
\cline{1-6}
$ Fi_{23} $ & $ \chi_{8}(1)=106743 $ & $ \chi_{94}(1)=504627200 $ & $ 7A $ & $ 2A $  & $ 14A $ \\
\cline{1-6}
$ Fi_{24}' $ & $ \chi_{4}(1)=249458 $ & $ \chi_{97}(1)=197813862400 $ & $ 11A $ & $ 2A $  & $ 22A $ \\
\cline{1-6}
$ HN $ & $ \chi_{2}(1)=133 $ & $ \chi_{46}(1)=3424256 $ & $ 7A $ & $ 2A $  & $ 14A $  \\
\cline{1-6}
$ Th $ & $ \chi_{6}(1)=30628 $ & $ \chi_{22}(1)=4096000 $ & $ 13A $ & $ 3A $  & $ 39A $ \\
\cline{1-6}
$ M $ & $ \chi_{9}(1) $ & $ \chi_{102}(1) $ & $ 17A $ & $ 2A $  & $ 34A $ \\
\cline{1-6}
$ J_{1} $ & $ \chi_{9}(1)=120 $ & $ \chi_{2}(1)=56 $ & $ 3A $ & $ 2A $ & $ 6A $\\
\cline{1-6}
$ O'N $ & $ \chi_{8}(1)=32395 $ & $ \chi_{23}(1)=175616 $ & $ 5A $ & $ 2A $ & $ 10A $ \\
\cline{1-6}
$ J_{3} $ & $ \chi_{2}(1)=85 $  & $ \chi_{14}(1)=1920 $ & $ 5A $ & $ 2A $ & $ 10A $ \\
\cline{1-6}
$ Ly $ & $ \chi_{5}(1)=48174 $ & $ \chi_{7}(1)=120064 $ & $ 7A $ & $ 2A $ & $ 14A $ \\
\cline{1-6}
$ J_{4} $ & $ \chi_{9}(1)=1187145 $ & $ \chi_{53}(1)=1981808640 $ & $ 5A $ & $ 2A $ & $ 10A $ \\
\cline{1-6}
$ ^{2}\mathrm{F}_{4}(2)' $ &  $ \chi_{8}(1)=325 $ & $ \chi_{21}=2048 $ & $ 5A $ & $ 2A $ & $ 10A $ \\
\hline
\end{tabular}
\end{center}

For the sporadic simple groups in Table \ref{SporadicTable2}.2, we have two irreducible characters $ \theta_{1} $ and $ \theta_{2} $ of $ N $ of $ p_{1} $-defect zero and $ p_{2} $-defect zero, respectively. We also have two $ p_{i} $-elements of distinct orders for each $ \theta_{i} $ on which $ \theta_{i} $ vanishes. This is case (b). For the case when $ N\cong Suz $, we have that $ N $ has an irreducible character of $ 5 $-defect zero and elements of orders $ 5, 10 $ and $ 20 $. In the case when $ N\cong HS $, $ N $ has an irreducible character of $ 3 $-defect zero and elements of orders $ 3, 6 $ and $ 12 $. Both of these cases above satisfy case (d).

\begin{center}\label{SporadicTable2}
Table 3.2
\begin{tabular}{|c|c|c|c|c|c|c|}
\cline{1-7}
$ N $ & $ \theta_{1}(1) $ & $ \ord(v\mathcal{C}_{1}) $  & $ \ord(v\mathcal{C}_{2}) $ & $ \theta_{2}(1) $ &  $ \ord(v\mathcal{C}_{3}) $ & $ \ord(v\mathcal{C}_{4}) $ \\
\cline{1-7}
$ M_{12} $  &   $ \chi_{8}(1)=55 $ & $ 5A $    &  $ 10A $ & $ \chi_{7}(1)=54 $ & $ 3A $ & $ 6A $ \\
\cline{1-7}
$ M_{24} $  &   $ \chi_{3}(1)=45 $ & $ 5A $   &  $10A$ & $ \chi_{26}(1)=10395 $ & $ 3A $ & $ 6A $ \\
\cline{1-7}
$ J_{2} $  &   $ \chi_{18}(1)=225 $ & $ 5A $   &  $ 10A $ & $ \chi_{14}(1)=189 $ & $ 3A $ & $ 6A $ \\
\cline{1-7}
$ Co_{3} $  &   $ \chi_{6}(1)=896 $  & $ 7A $   &  $ 14A $ & $ \chi_{23}(1)=31625 $ & $ 5A $ & $ 10A $ \\
\cline{1-7}
$ Co_{1} $  &   $ \chi_{4}(1)=1771 $ & $ 11A $   &  $ 22A $  & $ \chi_{3}(1)=299 $ & $ 13A $ & $ 26A $ \\
\cline{1-7}
$ B $ & $ \chi_{11}(1)=3214743741 $ & $ 11A $ & $ 22A $ & $ \chi_{8}(1)=347643114 $ & $ 13A $ & $ 26A $\\
\cline{1-7}
$ Ru $ & $ \chi_{2}(1)=378 $ & $ 3A $ & $ 6A $ & $ \chi_{2}(1)=378 $ & $ 7A $ & $ 14A $ \\
\hline
\end{tabular}
\end{center}
We are left with the case when $ N\cong M_{22} $. There exists $ \theta_{1}=\chi_{10} $, $ \theta_{2}=\chi_{3} $ which are extendible to $ G $ by Lemma \ref{BCLP07Lemma5}. Note that $ \chi_{3}(3A)=\chi_{3}(6A)=0 $ and $ \chi_{10}(4A)=0 $ and hence this is case (c).

Suppose that $ N $ is an alternating group $ \A_{n} $, $ n\geqslant 5 $. Since $ N\notin \mathfrak{S} $, $ n\geqslant 8 $. Suppose $ N\cong \A_{8} $ or $ \A_{9} $. Then $ N $ has two irreducible characters $ \chi_{15} $ and $ \chi_{7} $ of $ 3 $-defect zero and $ 5 $-defect zero, respectively. The result follows because $ N $ has elements of orders $ 3 $, $ 5 $ and $ 15 $. This is case (a).

Suppose $ N\cong \A_{10} $. By Lemma \ref{pdefectzero5ormore}, $ N $ has irreducible characters $ \theta_{1} $ and $ \theta_{2} $, of $ 5 $-defect zero and $ 7 $-defect zero, respectively. Since $ N $ has elements of orders $ 5 $, $ 10 $, $ 7 $ and $ 21 $, the result follows using Lemma \ref{Bro16Lemma2.2}. This is case (b).

Suppose that $ N $ is an alternating group $ \A_{n} $, $ n\geqslant 11 $. By Lemma \ref{pdefectzero5ormore}, $ N $ has an irreducible character is of $ 5 $-defect zero. Note that $ N $ has three elements $ (12345) $, $ (12345)(6789)(10~11) $ and $ (12345)(67)(89) $ of orders $ 5$, $10 $ and $ 20 $, respectively. Using Lemma \ref{Bro16Lemma2.2}, we have three elements are vanishing elements of $ G $. This is case (d).

We may assume that $ N $ is simple group of Lie type. Suppose that $ N $ has an element of order $ 2s $, $ s $ an odd prime. Then $ N $ has irreducible characters (not necessarily distinct) $ \theta _{1} $ and $ \theta _{2} $ of $ 2 $-defect zero and of $ s $-defect zero by Lemma \ref{pdefectzero5ormore}. Since $ N $ has elements of orders $ 2 $, $ s $ and $ 2s $, the result follows since these two elements are vanishing elements by Lemma \ref{Bro16Lemma2.2}. This is case (a). We may assume that $ N $ has no element of order $ 2s $, $ s $ an odd prime. Then the centralizer of each involution contained in $ N $ is a $ 2 $-group. It follows from \cite[III, Theorem 5]{Suz61} that $ N $ is isomorphic to one of the following: $ \PSL_{2}(r) $, where $ r $ is a Fermat or Mersenne prime; $ \PSL_{2}(9) $; $ N\cong \PSL_{3}(4) $; $ N\cong ^{2}\!\!\mathrm{B}_{2}(2^{2n + 1}) $, $ n\geqslant 1 $.
 
Since $ N\notin \mathfrak{S} $, we may assume that 
$ N\cong \PSL_{2}(r) $, where $ r $ is a Fermat or Mersenne prime and $ r\geqslant 17 $ or $ N\cong ^{2}\!\!\mathrm{B}_{2}(2^{2n + 1}) $, $ n\geqslant 2 $. Suppose that $ N\cong \PSL_{2}(r) $, $ r\geqslant 17 $. Suppose that $ N\cong \PSL_{2}(17) $. Then $ G $ has vanishing elements of orders $ 2,4,3 $ and $ 9 $. If $ N\cong \PSL_{2}(31) $, then $ G $ has vanishing elements of orders $ 3,5 $ and $ 15 $. The result follows. Suppose that $ N\cong\PSL_{2}(r) $, where $ r \geqslant 127 $ is a Fermat or Mersenne prime. The character tables of $ \PSL_{2}(q) $ are well known (see for example \cite[Chapter 38]{Dor71}). In particular $ N $ has vanishing elements orders which are factors of $ (r+1)/2 $ and $ (r -1)/2 $. Since $ r $ is a Fermat or Mersenne prime either $ r - 1=2^{n} $ or $ r + 1=2^{n} $ for positive integer $ n $. This means that $ N $ has elements of orders $ 2 $ and $ 4 $. Hence either $ r + 1 $ or $ r -1 $ is divisible by $ 3 $. Since $ r \geqslant 157 $, we have that $ r\pm 1=3^{k} $, $ k\geqslant 5 $ or $ r \pm 1 $ is divisible by at least distinct primes. Hence $ N $ has vanishing elements of orders $ 2,4 ,3, 9$ or $ 3,q, 3q $ for some prime $ q\neq 3 $. In other words, this is case (b) or  case (a).

Suppose that $ N\cong ^{2}\!\!\mathrm{B}_{2}(2^{2n+1}) $ with $ n\geqslant 2 $. In this case, $ N $ has cyclic Hall subgroups $ H_{1} $, $ H_{2} $ and $ H_{3} $ with 
\begin{center}
$ |H_{1}|=2^{2n+1}-2^{n+1} + 1=r_{1} $, $ |H_{2}|=2^{2n+1} + 2^{n+1} + 1=r_{2} $ and $ |H_{3}|=2^{2n+1}-1=r_{3} $.
\end{center} 
Orders of elements of $ N $ consist of factors of $ r_{1} $, $ r_{2} $ and $ r_{3} $. Note that all the non-trivial elements of $ N $ are vanishing elements of $ G $. If $ r_{i} $ is not prime for some $ i=1,2,3 $, then the result follows since $ N $ has elements of order $ 2 $ and $ 4 $. This is case (b). Hence we may assume that $ r_{1} $, $ r_{2} $ and $ r_{3} $ are prime. This means $ |\pi(N)|=4 $. Using \cite{BCM01}, we have that $ N\cong ^{2}\!\!\mathrm{B}_{2}(32) $. Now $ N $ has elements of orders $ 2 $, $ 4 $, $ 5 $ and $ 25 $ which are vanishing elements of $ G $ and the result follows. This is case (b) and this concludes our argument.
\end{proof}

We are now ready to prove our reduction theorem below:

\begin{theorem}(\textbf{Reduction Theorem})\label{reduction}
Let $ G $ be a finite group and let $ p $ be a prime. If $ G $ satisfies property \eqref{ee:star} for $ p $, then there exists a solvable normal subgroup $ N $ of $ G $ such that one of the following holds:
\begin{itemize}
\item[(a)] $ G $ is solvable;
\item[(b)] $ p=2, 3 $ or $ 5 $, $ G/N \cong \A_{5} $ and one of the following holds:
\begin{itemize}
\item[(i)] $ N/\textbf{O}^{2}(N) $ is an elementary abelian group;
\item[(ii)] $ N/\textbf{O}_{\pi '}(G) $ is a $ \pi $-group with $ \pi =\{2,3,5 \} $, that is, $ \textbf{O}_{\pi '}(G) $ is a normal Hall $ \pi ' $-subgroup of $ G $;
\end{itemize} 
\item[(c)] $ p=2 $ and $ G/N\in \{\A_{6}, \A_{6}{:}2_{3}\} $;
\item[(d)] $ p=2, 3, 5 $ or $ 7 $ and $ G\cong \A_{7} $.
\item[(e)] $ p = 2 $ and $ G/N \cong \PSL_{2}(7)$;
\item[(f)] $ p = 3 $ and $G/N \cong \PSL_{2}(8)$;
\item[(g)] $ p = 2 $ and $G/N \cong \PSL_{3}(4)$;
\item[(h)] $ p = 2 $ and $G/N \cong ^{2}\!\!\mathrm{B}_{2}(8)$.
\end{itemize}
\end{theorem} 
\begin{proof}
Suppose that $G$ has no composition factor isomorphic to $ S\in \mathfrak{S} $. We prove that $G$ is solvable using induction on $|G|$. Let $N$ be a non-trivial normal subgroup of $G$.
Then $ G/N $ has no composition factor isomorphic to $S$ and satisfies property \eqref{ee:star} and hence $G/N$ is a solvable group. If $N_{1} $ and $N_{2} $ are two minimal normal subgroups of $G$, then $G/N_{1} $ and $G/N_{2} $ are solvable. Hence $G$ is solvable since the class of finite solvable groups is a formation. We may assume that G has a unique non-abelian minimal normal subgroup $N$. Then by Proposition \ref{vanishingordersaremore}, $G$ does not satisfy property \eqref{ee:star} for any prime $ p $.

We may assume that $G$ has a composition factor isomorphic to $ S \in \mathfrak{S} $. Since $G$ satisfies property \eqref{ee:star}, $\Gamma(G) $ is disconnected. By Theorem \ref{DPSS10aTheoremB}, $G$ has a unique non-abelian composition factor. Letting $ N $ to be the solvable radical of $ G $, we have that $ G/N $ is an almost simple group and there exists a normal subgroup $ M $ of $G$ such that $M/N$ is isomorphic to $ S \in \mathfrak{S} $.

Suppose that $ M = G $ and $ G/N \cong \mathrm{A}_{5} $. Then $ p = 2, 3 $ or $5$. By \cite[Lemma 2.2]{MNO00}, $G$ has a $2$-element, a $3$-element and a $5$-element which are vanishing elements. Let $\pi = \{2, 3, 5\} $. Suppose that there exist an $r$ element, $r \notin \pi $ prime, which is a vanishing element. Then $r$ is an isolated vertex in the vanishing prime graph of $G$ by property \eqref{ee:star}. In other words, $n(\Gamma(G)) \geqslant 4 > 3 = n(\Gamma(\A_{5}))$ by Theorem \ref{DPSS10aTheoremB}, a contradiction.
Hence if $G$ has another vanishing element other than $q$-elements, $q \in \pi $, then it has to be a $q$-singular element $g$ such that some prime $r \notin \pi $ divides $\ord(g)$, using property \eqref{ee:star}. This means that $G$ has a normal Sylow $r$-subgroup. Therefore either $\Van(G)$ consists of $2$-elements, $3$-elements and $5$-elements or $G$ has an additional vanishing element whose order is divisible by $ qr $, $q \in \pi $ and $ r \notin \pi $. The first case is (b)(i) by \cite[Theorem A]{ZLS11}.

For the latter case we have that $G$ has a normal Hall $ \pi' $-subgroup. This is case (b)(ii) of the theorem.

Suppose that $M < G$ and $G/N \cong \mathrm{S}_{5} $. Then there exist elements of orders $2$, $3$ and $6$ which are vanishing elements of $G/N$. Hence $G/N$ does not satisfy property \eqref{ee:star} and the result follows.

If $M/N \cong \mathrm{A}_{6}$, then $M/N$ has an irreducible character of $2$-defect zero and so vanishing elements of orders $2$ and $4$ by Lemma \ref{Bro16Lemma2.2}. This means that $G$ does not satisfy property \eqref{ee:star} for an odd prime $p$. Suppose that $p = 2$. Then checking the character tables in the \Atlas{} \cite{CCNPW85}, for an almost simple group $G/N$ such that $|G/M|\leqslant 2$, our result follows except for $ \mathrm{A}_{6} $ and $\mathrm{A}_{6} {:}2_{3} $. For $G$ such that $|G/M| = 4$, we have the character table from \GAP{} \cite{GAP17} and can conclude that $G$ does not satisfy property \eqref{ee:star}. Hence (c) follows.

Suppose that $M = G$ and $G/N \cong \mathrm{A}_{7}$. Then $G/N$ has vanishing elements of orders $3$, $4$, $5 $ and $7$. So $\Gamma(G/N) $ has four components. Suppose $N_{1} $ is a normal subgroup such that $N/N_{1} $ is a chief factor of $G$. In particular, $N/N_{1} $ is a $r$-group for some prime $r$. By \cite[Lemma 5.1]{DPSS10b}, $\Gamma(G/N_{1}) $ has at most two connected components. This means that $G/N_{1} $ does not satisfy property \eqref{ee:star}. We have that $G$ does not satisfy property \eqref{ee:star}. Hence $ p = 2, 3, 5 $ or $7$, $N = 1$ and this is case (d) of the theorem.

Suppose that $M < G $ and $ G/N \cong\mathrm{S}_{7} $. Then there exist elements of orders $ 2, 3 $ and $6$ which are vanishing elements of $G/N$ and the result follows.

Now if $M/N \in \{ \PSL_{2}(7), \PSL_{3}(4), ^{2}\!\mathrm{B}_{2}(8)\} $, then $M/N$ has an irreducible character of $2$-defect zero and elements of orders $2$ and $4$ and so $G/N$ does not satisfy property \eqref{ee:star} for an odd prime $p$. We may assume that $p = 2$.

Suppose that $M/N \cong \PSL_{2}(7)$ or $M/N \cong ^{2}\!\!\mathrm{B}_{2}(8)$. Then by consulting character tables in the \Atlas{} \cite{CCNPW85}, our result follows. Suppose $M/N \cong \PSL_{3}(4)$. Then for an almost simple $G$ such that $|G/M|\leqslant 6 $, our result follows by checking the explicit character tables in the \Atlas{} \cite{CCNPW85}. For the case when $|G/M| = 12$, we obtain the character table from \GAP{} \cite{GAP17} and our result follows.

Suppose that $G/N \cong \PSL_{2}(8) $. Then $p = 3$ and $G$ satisfies property \eqref{ee:star}. If $G/N \cong \PSL_{2}(8){:}3 $, then $G/N$ has vanishing elements of orders $2, 3$ and $6$ and the result follows.
\end{proof}

We prove the first part of Theorem A below. First we prove a lemma.

\begin{lemma}\label{everyphasavanishingelement}
Let $N$ be a non-abelian minimal normal subgroup of $G$. Suppose that $p$ is a prime such that $p \mid |N| $. Then there exists a vanishing element of $G$ that is a $p$-element.
\end{lemma}
\begin{proof}
Since $N$ is non-abelian, $N = S_{1}\times S_{2} \times \cdots \times S_{k} $, where $ S_{i} \cong S $, and $S$ is non-abelian simple group. Let $S$ be a finite group of Lie type or $ p \geqslant 5 $. By Lemma \ref{pdefectzero5ormore}, $S$ has an irreducible character $\theta $ of $p$-defect zero. Let $x$ be a $p$-element of $S$.
Then $\psi = \theta \times \theta \times \cdots \times \theta $ is an irreducible character of $N$ of $p$-defect zero. Hence $ g = x \times x \times \cdots \times x \in N $ is a $p$-element and by Lemma \ref{Bro16Lemma2.2}, $g$ is a vanishing element of $G$.

Suppose that $p \in \{2, 3\}$ and $S$ is a sporadic simple group or an alternating group $\mathrm{A}_{n}$, $n \geqslant 7$. By \cite[Lemma 2.3]{DPSS09} and \cite[Proposition 2.4]{DPSS09}, there exists an irreducible character $ \theta $ of $S$ that extends to $\Aut(S)$ and $\theta(x) = 0 $ for some $p$-element $ x \in S $. Let
$g = x \times x \times \cdots \times x \in N$ and note that $g$ is a $p$-element. By \cite[Proposition 2.2]{DPSS09}, $ \theta \times \theta \times \cdots \times \theta \in \Irr(N) $ extends to some $\chi \in \Irr(G)$. Hence $\chi(g) = (\theta(x))^{k} = 0 $ and this concludes our proof.
\end{proof}

\begin{theorem}\label{vanishingorderimpliessolvable} Let $G$ be a finite group and let $p$ be a prime. Suppose that $|\Vo_{p'}(G)| = 1$. Then $G$ is solvable.
\end{theorem}
\begin{proof}
We prove that $G$ is solvable using induction on $|G|$. Let $N$ be a non-trivial normal subgroup of $G$. Then $|\Vo_{p'}(G/N)|\leqslant 1$ and hence $G/N$ is a solvable group. If $N_{1} $ and $N_{2}$ are two minimal normal subgroups of $G$, then both $G/N_{1}$ and $G/N_{2}$ are solvable. Hence $G$ is solvable. We may assume that $G$ has a unique non-abelian minimal normal subgroup $N$. Note that $|\pi(N)| \geqslant 3 $ by \cite[Theorem 3.10]{Isa06}. Using Lemma \ref{everyphasavanishingelement}, we have that $ |\Vo_{p'}(G)| \geqslant 2$. Hence the result follows.
\end{proof}

\section{Solvable groups}

In this section we shall restate and prove our main theorems. We start with Theorem D.

%\begin{theorem}
%Let $ G $ be a finite non-abelian solvable group and $ p, q $ be distinct primes. If every irreducible character of $ G $ vanishes on elements of prime power order, then one of the following holds.
%\begin{itemize}\label{pqtheorem}
%\item[(i)] $ G $ is a $ p $-group;
%\item[(ii)] $ G/\textbf{Z}(G) $ is a Frobenius group with a Frobenius complement of $ p $-power order and $ \textbf{Z}(G)=\textbf{O}_{p}(G) $;
%\item[(iii)] $ G $ is a Frobenius group with a Frobenius complement of $ p $-power order and a kernel of $ q $-power order, that is, $ \pi(G)=\{p,q\} $;
%\item[(iv)] $ G $ is a nearly $ 2 $-Frobenius group and $ \pi(G)=\{p,q\} $.
%\end{itemize}
%\end{theorem}
\begin{proof}[\textbf{Proof of Theorem D}]
Suppose that $ \Gamma(G) $ has one component. Then all vanishing elements are $ p $-elements for some $ p $. By \cite[Theorem A]{BDS09}, $ G $ is either a $ p $-group or $ G/\textbf{Z}(G) $ is a Frobenius group with a Frobenius complement of $ p $-power order and $ \textbf{Z}(G)=\textbf{O}_{p}(G) $. This is (a) and (b).

Suppose $ \Gamma(G) $ is disconnected. By Theorem \ref{DPSS10bTheoremA}, $ \Gamma(G) $ has at exactly two connected components. This means that $ G $ has vanishing elements which are $ p $-elements and $ q $-elements for some distinct primes $ p $ and $ q $. Note that $ G $ is a Frobenius group or a nearly $ 2 $-Frobenius group.

Suppose that $ G $ is a Frobenius group. Let the Frobenius complement be of $ p $-power order for some prime $ p $. Suppose that the order of the kernel $ K $ is divisible by at least two distinct primes. There exists a $ q $-element which is a vanishing element of $ G $ contained in $ K $. Using \cite[Proposition 2.1]{DPSS10b}, there exists $ g $ which is a vanishing element such that $ \pi(\ord(g))=\pi(|K|) $, a contradiction. Therefore $ |K| $ is of $ q $-power order as expected in (c). If the Frobenius complement is divisible by two distinct primes, $ p $ and $ q $, then there exists a prime $ r $ that divides $ |K| $. Using \cite[Proposition 3.2]{DPSS10b}, there exists a vanishing element $ g $ such that $ \ord(g) $ is either divisible by $ pr $ or $ qr $, a contradiction. 

Suppose that $ G $ is a nearly $ 2 $-Frobenius group. Then there exist two normal subgroups $ F $ and $ L $ of $ G $ with the following properties: $ F=F_{1}\times F_{2} $ is nilpotent, where $ F_{1} $ and $ F_{2} $ are normal subgroups of $ G $. Furthermore, $ G/F $ is a Frobenius group with kernel $ L/F $, $ G/F_{1} $ is a Frobenius group with kernel $ L/F_{1} $, and $ G/F_{2} $ is a $ 2 $-Frobenius group. Since $ G/F_{2} $ is a $ 2 $-Frobenius group and $ G/F $ is a Frobenius group with kernel $ L/F $, it follows that $ L/F_{2} $ is a Frobenius group with kernel $ F/F_{2} $. By the argument above, $ |G/L| $ is a $ p $-power and $ |L/F_{1}| $ is a $ q $-power. If $ |F| $ is divisible by three distinct primes, then there exists a prime $ r $ different from $ p $ and $ q $. Using \cite[Proposition 3.2]{DPSS10b}, there exists a vanishing element $ g $ such that $ \ord(g) $ is either divisible by $ pr $ or $ qr $, a contradiction. Hence $ |F| $ is divisible by at most two primes $ p $ and $ q $. The result then follows.
\end{proof}

We note here that the converse of this theorem is not true (see \cite[Example 2]{BDS09}). The following example shows a group which satisfies properties in part (iv) of the theorem.

\subsection{Example}
Let $ G=((\mathrm{C}_{2}\times \mathrm{C}_{2})\rtimes \mathrm{C}_{9})\rtimes \mathrm{C}_{2} $, $ L=(\mathrm{C}_{2}\times \mathrm{C}_{2})\rtimes \mathrm{C}_{9} $, $ F_{1}=\mathrm{C}_{2}\times \mathrm{C}_{2} $, $ F_{2}= \mathrm{C}_{3} $ and $ F=(\mathrm{C}_{2}\times \mathrm{C}_{2})\times \mathrm{C}_{3} $. Then $ G/F\cong \mathrm{S}_{3} $ and $ G/F_{1}\cong \mathrm{D}_{18} $, a Frobenius group of with kernel $ L/F_{1}\cong \mathrm{C}_{9} $. Also $ G/F_{2}\cong \SSS_{4} $, a $ 2 $-Frobenius group. In other words, $ G $ is a nearly $ 2 $-Frobenius group and $ \mathrm{Vo}(G)=\{2,4,9\} $ using \GAP{} \cite{GAP17}.

\begin{theorem}\label{thmB}
Let $G$ be a finite solvable group and $p $ be a prime. If $G$ satisfies property \eqref{ee:star} for $p$, then $\textbf{O}^{pp'pp'}(G) = 1$.
\end{theorem}
\begin{proof}
We first consider when $\Gamma(G) $ is connected. Suppose that $\Gamma(G) $ consists of a single vertex. If the vertex is $q$, then by Theorem \ref{normalp-complementvanishing2}, $G$ is a $q$-group or $G/\mathbf{Z}(G) $ is a Frobenius group with a Frobenius complement of $q$-power order and $\textbf{Z}(G) = \textbf{O}_{q}(G)$. If $p = q$, then $ \textbf{O}^{pp'}(G) = 1$. If $ p \not= q$, then $G$ has a normal Sylow $p$-subgroup. Hence $\textbf{O}^{p}(\textbf{O}^{pp'}(G)) = 1$.

Suppose that $\Gamma(G) $ has at least two vertices. Then $ p \mid a $ for all $ a \in \Vo(G) $. Using Theorem \ref{normalp-complementvanishing}, $\textbf{O}^{pp'}(G) = 1$.

Suppose that $\Gamma(G) $ is disconnected. By Theorem \ref{DPSS10bTheoremA}, $\Gamma(G) $ has two connected components. Since $G$ satisfies property \eqref{ee:star} for $p$, there exists exactly one $b \in \Vo(G) $ such that $ p \nmid b$. If $b$ is divisible by two distinct primes $p_{1}, p_{2}$, then $G$ has a normal Sylow $p_{i}$-subgroup, $i = 1, 2 $ by \cite[Theorem A]{DPSS09} since G has no $p_{i}$-elements as vanishing elements. Hence $G$ has a normal Hall $p'$-subgroup $H$ and $H$ is a normal nilpotent $p$-complement of $G$.

We may assume that $b \in \Vo(G) $ and $b$ is a $q$-power for the rest of the proof. If $a = p^{m} $ for all $ a \in \Vo(G)\setminus \{b\}$, then by Theorem D, $G$ is a Frobenius group with a Frobenius complement of $r$-power order and a kernel of $s$-power order or $G$ is a nearly $2$-Frobenius group and $\pi(G) = \{r, s\}$. Assume that $G$ is a Frobenius group with a Frobenius complement of $r$-power order and a kernel $K$ of $s$-power order. If $r = p$, then the result follows. If $s = p$, then $\textbf{O}^{p}(G) = G$ and $\textbf{O}^{p'}(G) = K$. Hence $\textbf{O}^{p}(K) = 1 $ and the result follows. Assume that $G$ is a nearly $2$-Frobenius group and $\pi(G) = \{r, s\}$. Then there exist two normal subgroups $F$ and $L$ of $G$ with the following properties: $F = F_{1} \times F_{2} $ is nilpotent, where $F_{1}$ and $F_{2} $ are normal subgroups of $G$. Furthermore, $G/F$ is a Frobenius group with kernel $L/F$, $G/F_{1}$ is a Frobenius group with kernel $L/F_{1}$, and $G/F_{2}$ is a $2$-Frobenius group. In
particular, $|G/L|$ is an $r$-power, $|L/F_{1}|$ is an $s$-power and $|F_{1}|$ is an $r$-power. If $r = p$, then $\textbf{O}^{p}(G) \leqslant L$, $\textbf{O}^{p'}(L) \leqslant F_{1}$ and $\textbf{O}^{p}(F_{1}) = 1$. Assume that $s = p$. Then $\textbf{O}^{p}(G) = G$, $\textbf{O}^{p'}(G) \leqslant L$, $\textbf{O}^{p}(L) \leqslant F_{1}$ and $\textbf{O}^{p'}(F_{1}) = 1$.

Assume that there exist $a,c \in \Vo(G)$ such that $c = p^{m}$ and $c = p^{k}n$, where $k$, $ m $ and $n$ are positive integers and $n \geqslant 2$. Let $\pi = \{p, q\}$. Then the Hall $\pi'$-subgroup of $G$ is nilpotent and normal. By Theorem \ref{DPSS10bTheoremA}, $G$ is a Frobenius group or a nearly $2$-Frobenius group. Suppose that $G$ is a Frobenius group $KM$ with Frobenius kernel $K$ and Frobenius complement $M$. If $n \mid |M| $, let $r$ be a prime such that $r \mid n$. Then consider $r$ and $q$. If $s$ is a prime such that $s \mid |K|$, then by \cite[Proposition 3.2]{DPSS10a}, there exists an vanishing element $g$ of $G$ such that $\ord(g)$ either divisible by $rs$ or $qs$, a contradiction. Hence $|M|$ is divisible by two primes $p$ and $q$. Using Theorem \ref{Bro01Theorem1.4}, we have that $M$ has a unique normal subgroup $N$ such that all the Sylow subgroups of $N$ are cyclic and $M/N \in \{1, \mathrm{V}_{4}, \mathrm{A}_{4}, \mathrm{S}_{4}, \mathrm{A}_{5}, \mathrm{S}_{5}\} $. Since $G$ is solvable we need not consider the cases when $M/N \cong\mathrm{A}_{5}$ or $ \mathrm{S}_{5} $. Note that $N$ is metacyclic and supersolvable by \cite[p. 290]{Rob95}.

Suppose that $M/N \cong 1$. Let $P$ be a Sylow $p$-subgroup of $G$ and $Q$ be a Sylow $q$-subgroup of $G$. If $ p > q $, then by \cite[Theorems 6.2.5 and 6.2.2]{Bec71}, $P$ is a normal subgroup of $M$ and hence $KP$ is normal in $G$. Hence $\textbf{O}^{pp'pp'}(G) = 1$. If $q > p$, then by \cite[Theorems 6.2.5 and 6.2.2]{Bec71}, $Q$ is a normal subgroup of $M$ and so $K$Q is normal in $G$. Hence $\textbf{O}^{pp'}(G) = 1$. Assume that $M/N \cong \mathrm{V}_{4} $. So $\pi(M) = \{2, r\}$, where $r$ is an odd prime. Let $R$ be a Sylow $r$-subgroup of $G$. Since $N$ is supersolvable, $R$ is normal in $M$ by \cite[Theorems 6.2.5 and 6.2.2]{Bec71}. It follows that $KR$ is normal in $G$. If $p = 2$, then $\textbf{O}^{pp'}(G) = 1$. If $p = r$, then $\textbf{O}^{pp'pp'}(G) = 1$. Assume that $M/N \cong \mathrm{A}_{4}$. Then $\pi(M) = \{2, 3\}$. If $p = 2$, then $N$ is a $2$-group since the Sylow $3$-subgroup of $M$ is cyclic, and $\textbf{O}^{pp'pp'}(G) = 1$. If $p = 3$, then $N$ is a $3$-group. This is because if a Sylow $2$-subgroup $T$ has order greater than $4$, then $T$ is a generalized quaternion and hence has elements of order $2$ and $4$ which are vanishing elements of $G$. Therefore $\textbf{O}^{pp'pp'}(G) = 1$, as required. Assume that $M/N \cong\mathrm{S}_{4}$. Then $\pi(M) = \{2, 3\}$, $p = 2$ and $N$ is a $2$-group. Hence $\textbf{O}^{pp'pp'}(G) = 1$.

Suppose that $G$ is a nearly $2$-Frobenius group. Then there exist two normal subgroups $F$ and $L$ of $G$ with the following properties: $F = F_{1} \times F_{2} $ is nilpotent, where $F_{1}$ and $F_{2}$ are normal subgroups of $G$. Furthermore, $G/F$ is a Frobenius group with kernel $L/F$, $G/F_{1}$ is a Frobenius group with kernel $L/F_{1}$, and $G/F_{2}$ is a $2$-Frobenius group. Since $G/F_{2}$ is a $2$-Frobenius group and $G/F$ is a Frobenius group with kernel $L/F$, it follows that $L/F_{2}$ is a Frobenius group with kernel $F/F_{2}$. By \cite[Remark 1.2]{DPSS10a}, $G/L$ is cyclic and $L/F$ is cyclic with $|L/F|$ odd. Since $G/F_{1}$ is a Frobenius group with kernel $L/F_{1}$, we have $\gcd(|G/L|,|L/F_{1}|) = 1$. Since $G \setminus F \subseteq \Van(G)$, we have that $|G/L|$ is an $r$-power and $|L/F|$ is $s$-power, where $r$ and $s$ are primes. If $p = r$, then $\textbf{O}^{p}(G) \leqslant L$, $\textbf{O}^{p'}(L) \leqslant F_{1}$. Hence $\textbf{O}^{pp'pp'}(G) = 1$. If $p = s$, then $\textbf{O}^{pp'}(G) \leqslant L$. Note that the Hall $p'$-subgroup $H$ of $L$ is normal in $L$. Then $\textbf{O}^{p}(L) \leqslant H$. Hence $\textbf{O}^{p'}(H) = 1$. Suppose that $p \nmid a$ for all $a \in \Vo(G)$. Then $\Vo(G) = \{a, b\}$. Either $a$ or $b$ is a prime power by \cite[Theorem B]{MNO00}. Assume that $a$ is an $r$-power. If $b$ is not a prime power, then the Hall $r'$-subgroup is nilpotent and normal in $G$. It follows that $\textbf{O}^{pp'pp'}(G) = 1$.

Assume that there exists $c \in \Vo(G)$ such that $c = p^{k}n$, where $ k $ and $n$ are positive integers and $n \geqslant 2$ and there is no $ a \in \Vo(G)$ such that $a = p^{m}$. Then the Hall $q'$-subgroup is normal and nilpotent, and the result follows. 

Suppose that $p \nmid a $ for all $ a \in \Vo(G)$. Then $\Vo(G) =\{a, b\}$ where $ \gcd(a,b)=1 $. Either $a$ or $b$ is a prime power by \cite[Theorem B]{MNO00}. We may assume that $ b $ is a $ q $-power without loss of generality. Suppose further that $a$ is an $r$-power for some prime $ r $. The Hall $\{q, r\}'$-subgroup $H$ is nilpotent and normal in $G$. Then $\textbf{O}^{pp'}(G) \leqslant H$ and hence $\textbf{O}^{pp'pp'}(G) = 1$. If $a$ is not a prime power, then the Hall $q'$-subgroup is nilpotent and normal in $G$. It follows that $\textbf{O}^{pp'pp'}(G) = 1$. 
\end{proof}

\begin{proof}[\textbf{Proof of Theorem C}]
Since $ p > 7 $ in the hypothesis of Theorem C, $ G $ is solvable using Theorem \ref{reduction}. Now Theorem C follows from Theorem \ref{thmB}.
\end{proof}

%\begin{theorem}
%Let $G$ be a finite group and let $p $ and $ q$ be distinct primes. Suppose that $|\Vo_{p'}(G)| = 1$. Then G is solvable. Moreover, suppose that one of the following holds:
%\begin{itemize}
%\item[(a)] $ \Vo_{p'}(G)=\{b\} $ such that $ b $ is divisible by at least two primes;
%\item[(b)] $ \Vo_{p'}(G)=\{q^{n}\} $ for some positive integer $ n $ and $ Q' $ is subnormal, where $ Q $ is a Sylow $ q $-subgroup of $ G $.
%\end{itemize}
%Then $\textbf{O}^{pp'pp'}(G) = 1 $.
%\end{theorem}
\begin{proof}[\textbf{Proof of Theorem A}]
By Theorem \ref{vanishingorderimpliessolvable}, $G$ is solvable. Suppose that (a) holds. Then since there are no $ q $-elements that are vanishing elements for all $ q\neq p $, the Hall $ p' $-subgroup of $ G $ is normal and nilpotent by \cite[Theorem A]{DPSS09}. Hence $\textbf{O}^{pp'}(G) = 1 $

Suppose that (b) holds. Suppose that $\Gamma(G) $ is connected. Then $\Vo(G)=\{q^{n}\}$ or $\Vo(G)=\{q^{n}\} \cup X $, where $ X $ is a finite subset of  $ \{c\in \mathbb{N}: pq\mid c\} $. Then the Hall $q'$-subgroup $H$ is a normal and nilpotent subgroup of $G$ using \cite[Theorem A]{DPSS09}. So $\textbf{O}^{pp'}(G) \leqslant H$. It follows that $\textbf{O}^{pp'pp'}(G) = 1$. 

We may assume that $\Vo(G)= X \cup \{q^{n}\}\cup Y$, where $ X $ and $ Y $ are finite subsets of $\{a\in \mathbb{N}: a $ is a $p$-power$\} $ and  $ \{c\in \mathbb{N}: pq\mid c\}$, respectively. Let $ \pi=\{ p,q\} $ and suppose that $ P $ is a Sylow $ p $-subgroup of $ G $. Then the Hall $\pi'$-subgroup $H$ is a normal and nilpotent subgroup of $G$. Then $ \overline{G}=G/N $ is a $ \pi $-group, where $ N=H\textbf{O}_{p}(G)\textbf{O}_{q}(G) $. Since $ Q' $ is subnormal, we have that $ Q'\leqslant \textbf{O}_{q}(G) $ and $ \overline{G} $ is a product of an abelian $ q $-group $ \overline{Q}=QN/N $ and a $ p $-group $ \overline{P}=PN/N $. By Lemma \ref{FdGHN91Lemma3}, $ \overline{Q}\textbf{F}(\overline{G})  $ is a normal subgroup of $ \overline{G} $. Then $ \overline{G}/\overline{Q}\textbf{F}(\overline{G}) $ is a $ p $-group and $ \overline{Q}\textbf{F}(\overline{G})/\textbf{O}_{p}(\overline{G}) $ is a $ q $-group. Consider $ \textbf{O}_{p}(\overline{G})=P_{1}N/N $. Then $ P_{1}N/H\textbf{O}_{q}(G) $ is a $ p $-group and $ H\textbf{O}_{q}(G) $ is a $ p' $-group. Hence $\textbf{O}^{pp'pp'}(G) = 1$. 

We may assume that $\Gamma(G)$ is disconnected. Suppose that for all $a \in \Vo(G)\setminus \{q^{n}\}$, $a$ is not a $p$-power. Then the Hall $q'$-subgroup $H$ is a normal and nilpotent subgroup of $G$. Therefore $\textbf{O}^{pp'pp'}(G) = 1$.

Suppose that for all $a \in \Vo(G)\setminus \{q^{n}\}$, $a$ is a $p$-power. Then by Theorem D, $G$ is a Frobenius group with a Frobenius complement of $r$-power order and a kernel of $s$-power order or $G$ is a nearly $2$-Frobenius group and $\pi(G) = \{r, s\}$. Assume that $G$ is a Frobenius group with a Frobenius complement of $r$-power order and a kernel $K$ of $s$-power order. If $r = p$, then $\textbf{O}^{pp'}(G) = 1$. If $s = p$, then $\textbf{O}^{p}(G) = G$ and $\textbf{O}^{p'}(G) = K$. Hence $\textbf{O}^{p}(K) = 1 $ and the result follows. Assume that $G$ is a nearly $2$-Frobenius group and $\pi(G) = \{r, s\}$. Then there exist two normal subgroups $F$ and $L$ of $G$ with the following properties: $F = F_{1} \times F_{2} $ is nilpotent, where $F_{1}$ and $F_{2} $ are normal subgroups of $G$. Furthermore, $G/F$ is a Frobenius group with kernel $L/F$, $G/F_{1}$ is a Frobenius group with kernel $L/F_{1}$, and $G/F_{2}$ is a $2$-Frobenius group. In
particular, $|G/L|$ is an $r$-power, $|L/F_{1}|$ is an $s$-power and $|F_{1}|$ is an $r$-power. If $r = p$, then $\textbf{O}^{p}(G) \leqslant L$, $\textbf{O}^{p'}(L) \leqslant F_{1}$ and $\textbf{O}^{p}(F_{1}) = 1$. Assume that $s = p$. Then $\textbf{O}^{p}(G) = G$, $\textbf{O}^{p'}(G) \leqslant L$, $\textbf{O}^{p}(L) \leqslant F_{1}$ and $\textbf{O}^{p'}(F_{1}) = 1$.

Suppose that there exist $a,b \in \Vo(G)\setminus \{q^{n}\}$ such that $a = p^{m}$ and $b = p^{k}t$, where $k$, $ m $ and $t$ are positive integers and $t \geqslant 2$. Let $\pi = \{p, q\}$. Using \cite[Theorem A]{DPSS09}, the Hall $\pi'$-subgroup of $G$ is a nilpotent and normal subgroup of $ G $. By Theorem \ref{DPSS10bTheoremA}, $G$ is a Frobenius group or a nearly $2$-Frobenius group. Suppose that $G$ is a Frobenius group $KM$ with Frobenius kernel $K$ and Frobenius complement $M$. If $t \mid |M| $, let $r$ be a prime such that $r \mid t$. Since $ M\setminus \{1\}\subseteq \Van(G) $, $ G $ has some $ r $-elements which are vanishing elements. If $ r\neq q $, then  $|\Vo_{p'}(G)|\geqslant 2$, a contradiction. So $ r=q $ and $\Gamma(G)$ is connected, another contradiction from our hypothesis.

Suppose that $G$ is a nearly $2$-Frobenius group. Then there exist two normal subgroups $F$ and $L$ of $G$ with the following properties: $F = F_{1} \times F_{2} $ is nilpotent, where $F_{1}$ and $F_{2}$ are normal subgroups of $G$. Furthermore, $G/F$ is a Frobenius group with kernel $L/F$, $G/F_{1}$ is a Frobenius group with kernel $L/F_{1}$, and $G/F_{2}$ is a $2$-Frobenius group. Since $G/F_{2}$ is a $2$-Frobenius group and $G/F$ is a Frobenius group with kernel $L/F$, it follows that $L/F_{2}$ is a Frobenius group with kernel $F/F_{2}$. By \cite[Remark 1.2]{DPSS10a}, $G/L$ is cyclic and $L/F$ is cyclic with $|L/F|$ odd. Since $G/F_{1}$ is a Frobenius group with kernel $L/F_{1}$, we have $\gcd(|G/L|,|L/F_{1}|) = 1$. Since $G \setminus F \subseteq \Van(G)$, we have that $|G/L|$ is an $r$-power and $|L/F|$ is $s$-power, where $r$ and $s$ are primes. If $p = r$, then $\textbf{O}^{p}(G) \leqslant L$, $\textbf{O}^{p'}(L) \leqslant F_{1}$. Hence $\textbf{O}^{pp'pp'}(G) = 1$. If $p = s$, then $\textbf{O}^{pp'}(G) \leqslant L$. Note that the Hall $p'$-subgroup $H$ of $L$ is normal in $L$. Then $\textbf{O}^{p}(L) \leqslant H$. Hence $\textbf{O}^{p'}(H) = 1$ and the result follows. 
\end{proof}

%\begin{theorem} Let $ G $ be a finite solvable group. If $ |\Vo_{p}(G)|=1 $, then $ P' $ is subnormal in $ G $. In particular, $ P/\textbf{O}_{p}(G) $ is cyclic.
%\end{theorem}
\begin{proof}[\textbf{Proof of Theorem B}]
We consider the vanishing prime graph $ \Gamma(G) $ of $ G $. Suppose that $ \Gamma(G) $ is connected. Assume that $ \Gamma(G) $ consists of a single vertex. By Theorem \ref{normalp-complementvanishing2}, $G$ is a $p$-group or $G/\mathbf{Z}(G) $ is a Frobenius group with a Frobenius complement of $p$-power order and $\textbf{Z}(G) = \textbf{O}_{p}(G)$. If $ G $ is a $ p $-group then the result follows. We may assume that $G/\mathbf{Z}(G) $ is a Frobenius group with a Frobenius complement of $p$-power order and $\textbf{Z}(G) = \textbf{O}_{p}(G)$. If $ p $ is odd, then Sylow $ p $-subgroups are abelian and $ P/\textbf{O}_{p}(G) $ is cyclic. If $ p=2 $, then $ P/\textbf{O}_{p}(G) $ is either a generalized quaternion or cyclic since it is isomorphic to a Frobenius complement. If $ P/\textbf{O}_{p}(G) $ is a generalized quaternion, then $ G $ has vanishing elements of distinct orders divisible by $ 2 $ and $ 4 $, a contradiction. Hence our result follows.

If $ \Gamma(G) $ has more than one vertex, then it follows that the $ p $-elements $ G $ are non-vanishing. Hence $ P $ is normal in $ G $.

We may assume that $ \Gamma(G) $ is disconnected. By Theorem \ref{DPSS10bTheoremA}, $ G $ is either a Frobenius or a nearly $ 2 $-Frobenius group. Also note that $ p $ is an isolated vertex.

Suppose that $ G $ is a Frobenius group with a kernel $ K $ and a complement $ H $. If $ p\mid |K| $, then $ P $ is normal in $ G $ and the result follows. So $ p\mid |H| $. Then either $ P $ is a cyclic or a generalized quaternion. If $ P $ is generalized quaternion, then $ G $ contains vanishing elements of orders $ 2 $ and $ 4 $, a contradiction. Then $ P $ is cyclic as required. 

Suppose that $ G $ is a nearly $ 2 $-Frobenius group. Then there exist two normal subgroups $ F $ and $ L $ of $ G $ with the following properties: $ F=F_{1}\times F_{2} $ is nilpotent, where $ F_{1} $ and $ F_{2} $ are normal subgroups of $ G $. Furthermore, $ G/F $ is a Frobenius group with kernel $ L/F $, $ G/F_{1} $ is a Frobenius group with kernel $ L/F_{1} $, and $ G/F_{2} $ is a $ 2 $-Frobenius group. Note that $ G/L $ and $ L/F $ are both cyclic and $ F $ is nilpotent. Hence the result follows.
\end{proof}

\section{Acknowledgements}
The author would like thank the referee for the careful reading of this article and their comments.

\end{document}